\documentclass[reqno]{amsart}
\usepackage{amssymb,latexsym}
\usepackage{graphicx}
\usepackage{fancyhdr}
\numberwithin{equation}{section}
\newtheorem{thm}{Theorem}[section]
\newtheorem{theorem}[thm]{Theorem}

\newtheorem{definition}[thm]{Definition}
\newtheorem{proposition}[thm]{Proposition}
\newtheorem{corollary}[thm]{Corollary}

\begin{document}

\setcounter{page}{1}

\title[generalized $L$-functions]{On the special values of Berndt's generalized $L$-function}
\author{Yuan He}
\address{School of Mathematics and Big Data, Neijiang Normal University, Neijiang 641100, Sichuan, People's Republic of China}
\email{hyyhe@aliyun.com}

\begin{abstract}
In this paper, we study the generalized $L$-function considered by Berndt (1975). We introduce the generalized Apostol-Bernoulli polynomials and the generalized Apostol-Bernoulli functions, and establish some properties for them, including the Fourier series for these functions. We show that the values of Berndt's generalized $L$-function at integers are explicitly evaluated in terms of the generalized Apostol-Bernoulli functions.
\end{abstract}

\maketitle

\section{Introduction}\label{sec1}

Let $\mathbb{N}$ be the set of positive integers, $\mathbb{N}_{0}$ the set of non-negative integers, $\mathbb{Z}$ the set of integers, $\mathbb{R}$ the set of real numbers, and $\mathbb{C}$ the set of complex numbers. For $s\in\mathbb{C}$ and for a Dirichlet character $\chi$ modulo $q\in\mathbb{N}$, the Dirichlet
$L$-function is defined for $\Re(s)>1$ by the series
\begin{equation}\label{eq1.1}
L(s,\chi)=\sum_{n=1}^{\infty}\frac{\chi(n)}{n^{s}}.
\end{equation}
This function has a meromorphic continuation to $\mathbb{C}$. It is holomorphic if $B_{0}(\chi)=0$
(i.e., $\chi$ is non-principal), and has a simple pole at $s=1$ with residue $B_{0}(\chi)$
otherwise (i.e., if $\chi$ is the principal character), where $B_{n}(\chi)$ denotes the generalized Bernoulli numbers defined by the generating function of Leopoldt \cite{leopoldt}
\begin{equation}\label{eq1.2}
\sum_{r=1}^{q}\frac{\chi(r)te^{rt}}{e^{qt}-1}=\sum_{n=0}^{\infty}B_{n}(\chi)\frac{t^{n}}{n!}\quad(|t|<2\pi/q).
\end{equation}

An approach to the study of the Dirichlet $L$-function is based on the use of the relation that for $s\in\mathbb{C}$,
\begin{equation}\label{eq1.3}
L(s,\chi)=\frac{1}{q^{s}}\sum_{r=1}^{q}\chi(r)\zeta\biggl(s,\frac{r}{q}\biggl),
\end{equation}
where $\zeta(s,a)$ is the Hurwitz zeta-function defined for $\Re(s)>1$, $0<a\leqslant1$ by
\begin{equation}\label{eq1.4}
\zeta(s,a)=\sum_{n=0}^{\infty}\frac{1}{(n+a)^{s}}.
\end{equation}
Building on Riemann's first proof \cite{riemann} of the functional equation for the Riemann zeta-function, Hurwitz \cite{hurwitz} in 1882 found the functional equation for the Hurwitz zeta-function, as follows,
\begin{equation}\label{eq1.5}
\zeta\biggl(1-s,\frac{h}{k}\biggl)=\frac{2\Gamma(s)}{(2\pi k)^{s}}\sum_{r=1}^{k}\cos\biggl(\frac{\pi s}{2}-\frac{2\pi rh}{k}\biggl)\zeta\biggl(s,\frac{r}{k}\biggl),
\end{equation}
where $s\in\mathbb{C}$, $h,k\in\mathbb{N}$ with $1\leqslant h\leqslant k$, and $\Gamma(s)$ is the gamma function. Obviously, the case $h=k=1$ in \eqref{eq1.5} reduces to Riemann's functional equation for the Riemann
zeta-function:
\begin{equation}\label{eq1.6}
\zeta(1-s)=\frac{2\Gamma(s)}{(2\pi)^{s}}\cos\biggl(\frac{\pi s}{2}\biggl)\zeta(s)\quad(s\in\mathbb{C}).
\end{equation}
As part of his proof of \eqref{eq1.5}, Hurwitz \cite{hurwitz} also proved that if $0<a\leqslant1$ and $\Re(s)>1$, or if $0<a<1$ and $\Re(s)>0$, then
\begin{equation}\label{eq1.7}
\zeta(1-s,a)=\frac{\Gamma(s)}{(2\pi)^{s}}\bigl(e^{-\frac{\pi \mathrm{i}s}{2}}F(a,s)+e^{\frac{\pi \mathrm{i}s}{2}}F(-a,s)\bigl),
\end{equation}
where $\mathrm{i}^{2}=-1$, and $F(x,s)$ is the periodic zeta-function defined for $x\in\mathbb{R}$, $\Re(s)>1$ by
\begin{equation}\label{eq1.8}
F(x,s)=\sum_{n=1}^{\infty}\frac{e^{2\pi \mathrm{i}nx}}{n^{s}}.
\end{equation}
It should be noted that the series \eqref{eq1.8} converges absolutely for $\Re(s)>1$, and, if $x\in\mathbb{R}\setminus\mathbb{Z}$, it converges conditionally for $\Re(s)>0$. In fact, using Hurwitz's formula \eqref{eq1.7}, we can recover that the values of $\zeta(s,a)$ at non-positive integers are expressed in terms of the Bernoulli polynomial of degree $n$:
\begin{equation}\label{eq1.9}
\zeta(1-n,a)=-\frac{B_{n}(a)}{n}\quad(n\in\mathbb{N}).
\end{equation}
So if we take $s=1-n$ in \eqref{eq1.3} and then use \eqref{eq1.9}, we reobtain the result of Leopoldt \cite{leopoldt}:
\begin{equation}\label{eq1.10}
L(1-n,\chi)=-\frac{B_{n}(\chi)}{n}\quad(n\in\mathbb{N}).
\end{equation}
In 1970, Apostol \cite{apostol3} used \eqref{eq1.3} and Hurwitz's formula \eqref{eq1.7} to reprove the functional equation for the Dirichlet $L$-function:
\begin{equation}\label{eq1.11}
L(1-s,\chi)=\frac{q^{s-1}\Gamma(s)}{(2\pi)^{s}}\bigl(e^{-\frac{\pi \mathrm{i}s}{2}}+\chi(-1)e^{\frac{\pi \mathrm{i}s}{2}}\bigl)G(1,\chi)L(s,\overline{\chi}),
\end{equation}
where $s\in\mathbb{C}$, $\chi$ is a primitive character, $\overline{\chi}$ denotes the complex conjugation of $\chi$, and $G(n,\chi)$ is the Gauss sum defined by
\begin{equation*}
G(n,\chi)=\sum_{r=1}^{q}\chi(r)e^{\frac{2\pi \mathrm{i}rn}{q}}\quad(n\in\mathbb{Z}).
\end{equation*}
From \eqref{eq1.10} and \eqref{eq1.11}, it is easy to see that for a primitive character $\chi$, the values of $L(s,\chi)$ at positive integers can be given by the generalized
Bernoulli numbers (see, e.g., \cite[pp. 442--443]{neukirch}). In 2019, the author \cite{he1} used \eqref{eq1.3} and Hurwitz's functional equation \eqref{eq1.5} to show that the values of $L(s,\chi)$ at positive integers can be expressed in terms of the Bernoulli polynomials and the Gauss sum. We here mention that seven methods of proving Riemann's functional equation \eqref{eq1.6} can be found in Titchmarsh's \cite{titchmarsh} monograph, and a simple proof of Hurwitz's formula \eqref{eq1.7} was provided by Berndt \cite{berndt1} using the Euler-Maclaurin summation formula.

We now turn our attention to the extensions of the series \eqref{eq1.1}, \eqref{eq1.4} and \eqref{eq1.8}. In 1887, Lerch \cite{lerch} studied what is nowadays well-known as the Lerch zeta-function, namely
\begin{equation}\label{eq1.12}
\phi(x,a,s)=\sum_{n=0}^{\infty}\frac{e^{2\pi \mathrm{i}nx}}{(n+a)^{s}},
\end{equation}
where $x\in\mathbb{R}$, $a\in\mathbb{C}\setminus\{0,-1,-2,\ldots\}$, $\Re(s)>1$ if $x\in\mathbb{Z}$, and $\Re(s)>0$ otherwise, and used the contour integral
technique appearing in Riemann's first proof of \eqref{eq1.6} to prove that
if $0<x<1$ and $0<a\leqslant1$, then for all $s\in\mathbb{C}$,
\begin{eqnarray}\label{eq1.13}
&&\phi(x,a,1-s)\nonumber\\
&&\qquad=\frac{\Gamma(s)}{(2\pi)^{s}}\bigl(e^{\frac{\pi \mathrm{i}s}{2}-2\pi\mathrm{i}ax}\phi(-a,x,s)+e^{-\frac{\pi \mathrm{i}s}{2}+2\pi\mathrm{i}a(1-x)}\phi(a,1-x,s)\bigl).
\end{eqnarray}
It is worth noting that Lerch's functional equation \eqref{eq1.13} implies Hurwitz's formula \eqref{eq1.7}; see the 1951 work of Apostol \cite{apostol1}. In the same year, Apostol \cite{apostol2} reproved Lerch's functional equation \eqref{eq1.13} along the line of Riemann's second proof of \eqref{eq1.6}, and determined that the values of $\phi(x,a,s)$ at non-positive integers can be expressed by
\begin{equation}\label{eq1.14}
\phi(x,a,1-n)=-\frac{\beta_{n}(a,e^{2\pi \mathrm{i}x})}{n}\quad(n\in\mathbb{N}),
\end{equation}
where $\beta_{n}(x,\lambda)$ denotes the Apostol-Bernoulli polynomials, defined for $\lambda\in\mathbb{C}\setminus\{0\}$, by the generating function
\begin{equation}  \label{eq1.15}
\frac{te^{xt}}{\lambda e^{t}-1}=\sum_{n=0}^{\infty}
\beta_{n}(x,\lambda)\frac{t^{n}}{n!}\quad(|t+\log\lambda|<2\pi).
\end{equation}
Note that in \eqref{eq1.15}, the principal branch of the logarithm satisfies
\begin{equation*}
\log \lambda=\log|\lambda|+\mathrm{i}\arg \lambda \quad(-\pi<\arg \lambda\leqslant\pi).
\end{equation*}
Clearly, setting $\lambda=1$ in \eqref{eq1.15} gives the generating function of the Bernoulli polynomials. In 1972, Berndt \cite{berndt2} gave two new proofs of Lerch's functional equation \eqref{eq1.13}, one using contour integration and the other using the Euler-Maclaurin summation formula. Subsequently, Berndt \cite{berndt3} in 1975 considered the generalized $L$-function, which is defined by the series
\begin{equation}\label{eq1.16}
L(s,x,a,\chi)=\sideset{}{'}\sum_{n=0}^{\infty}\frac{\chi(n)e^{\frac{2\pi \mathrm{i}nx}{q}}}{(n+a)^{s}},
\end{equation}
where $q\in\mathbb{N}$ with $q\geqslant2$, $x\in\mathbb{R}$, $a\in\mathbb{C}$, $\chi$ is a Dirichlet character $\chi$ modulo $q$, the dash indicates that the term corresponding to $n=-a$ is omitted if $a\in\{0,-1,-2,\ldots\}$, $\Re(s)>1$ if $x\in\mathbb{Z}$ and $\gcd(x,q)=1$, and $\Re(s)>0$ otherwise. In particular, Berndt \cite[Theorem 5.1]{berndt3} showed that for $0\leqslant a\leqslant 1$, $L(s,x,a,\chi)$ can be analytically continued to $\mathbb{C}$ except when $x\in\mathbb{Z}$ and $\gcd(x,q)=1$, in which case $L(s,x,a,\chi)$ is analytic everywhere except $s=1$ (where it has a simple pole with residue $G(x,\chi)/q$), and he used his character analogues of the Euler-Maclaurin summation formula to prove that if $0\leqslant x\leqslant 1$, $0\leqslant a\leqslant 1$ and $\chi$ is a primitive character, then for all $s\in\mathbb{C}$,
\begin{eqnarray}\label{eq1.17}
L(1-s,x,a,\chi)&=&\frac{q^{s-1}\Gamma(s)}{(2\pi)^{s}}G(1,\chi)e^{-\frac{2\pi \mathrm{i}ax}{q}}\nonumber\\
&&\times\bigl(e^{-\frac{\pi \mathrm{i}s}{2}}L(s,a,-x,\overline{\chi})+\chi(-1)e^{\frac{\pi \mathrm{i}s}{2}}L(s,-a,x,\overline{\chi})\bigl).
\end{eqnarray}
Meanwhile, Berndt \cite[Definition 1]{berndt3} introduced the generalized Bernoulli functions, and used them in \eqref{eq1.17} to determine that if $\chi$ is a primitive character, then for $n\in\mathbb{N}$ and $0\leqslant a<1$,
\begin{equation}\label{eq1.18}
L(1-n,0,a,\chi)=-\frac{B_{n}(a,\chi)}{n},
\end{equation}
where $B_{n}(x,\chi)$ are the generalized Bernoulli polynomials, defined for a Dirichlet character $\chi$ modulo $q\in\mathbb{N}$, by the generating function
\begin{equation}\label{eq1.19}
\sum_{r=1}^{q}\frac{\chi(r)te^{(r+x)t}}{e^{qt}-1}=\sum_{n=0}^{\infty}B_{n}(x,\chi)\frac{t^{n}}{n!}\quad(|t|<2\pi/q).
\end{equation}

A natural question is whether, for a general $x$, the values of $L(s,x,a,\chi)$ at integers can be expressed by a sequence of polynomials. To this end, based on the work of Apostol \cite{apostol2}, Leopoldt \cite{leopoldt} and Berndt \cite{berndt3}, we construct the generalized Apostol-Bernoulli polynomials and the generalized Apostol-Bernoulli functions (see Definitions \ref{def2.1} and \ref{def2.7} below). We establish some basic properties for them and obtain the Fourier series for the generalized Apostol-Bernoulli functions. As applications, we show that the values of $L(s,x,a,\chi)$ at integers are explicitly expressed in terms of the generalized Apostol-Bernoulli functions.

This paper is organized as follows. Section \ref{sec2} presents the definitions of the generalized Apostol-Bernoulli polynomials and the generalized Apostol-Bernoulli functions, and establishes some basic properties of these polynomials and functions. Section \ref{sec3} presents the Fourier series of the generalized Apostol-Bernoulli functions, and uses it to express the values of $L(s,x,a,\chi)$ at integers and obtain a symmetric identity for the generalized Apostol-Bernoulli functions similar to that for the generalized Apostol-Bernoulli polynomials.

\section{Generalized Apostol-Bernoulli polynomials and functions}\label{sec2}

In order to explore the values of Berndt's generalized $L$-functions at integers, we first introduce the generalized Apostol-Bernoulli polynomials and the generalized Apostol-Bernoulli numbers.

\begin{definition}\label{def2.1} Let $n\in\mathbb{N}_{0}$ and $\lambda\in\mathbb{C}\setminus\{0\}$. For a Dirichlet character $\chi$ modulo $q\in\mathbb{N}$, the generalized Apostol-Bernoulli polynomials $\beta_{n,\chi}(x,\lambda)$ are defined by the generating function
\begin{equation}\label{eq2.1}
\sum_{r=1}^{q}\frac{\chi(r)\lambda^{r}te^{(r+x)t}}{\lambda^{q}e^{qt}-1}=\sum_{n=0}^{\infty}\beta_{n,\chi}(x,\lambda)\frac{t^{n}}{n!}\quad(|t+\log\lambda|<2\pi/q),
\end{equation}
where $\lambda^{z}=e^{z\log\lambda}$ for $z\in\mathbb{C}$ with $\log\lambda$ as defined in \eqref{eq1.15}. In particular, $\beta_{n,\chi}(\lambda)=\beta_{n,\chi}(0,\lambda)$ are called the generalized Apostol-Bernoulli numbers.
\end{definition}

It is trivial to see that the above Definition \ref{def2.1} contains the information about the Apostol-Bernoulli polynomials, the Apostol-Euler polynomials, the generalized Bernoulli polynomials and the generalized Euler polynomials. For example, setting $q=1$ in \eqref{eq2.1}, we have
\begin{equation}\label{eq2.2}
\beta_{n,\chi}(x-1,\lambda)=\lambda\beta_{n}(x,\lambda)\quad(n\in\mathbb{N}_{0}),
\end{equation}
and
\begin{equation}\label{eq2.3}
\frac{2\beta_{n+1,\chi}(x-1,-\lambda)}{n+1}=\lambda\varepsilon_{n}(x,\lambda)\quad(n\in\mathbb{N}_{0}),
\end{equation}
where $\varepsilon_{n}(x,\lambda)$ are the Apostol-Euler polynomials, defined for $\lambda\in\mathbb{C}\setminus\{0\}$, by the generating function (see, e.g., \cite{bayad2,luo2})
\begin{equation}\label{eq2.4}
\frac{2e^{xt}}{\lambda e^{t}+1}=\sum_{n=0}^{\infty}
\varepsilon_{n}(x,\lambda)\frac{t^{n}}{n!}\quad(|t+\log\lambda|<\pi).
\end{equation}
Clearly, the case $\lambda=1$ in \eqref{eq2.4} gives the generating function of the Euler polynomials $E_{n}(x)=\varepsilon_{n}(x,1)$.
It is also easily seen from \eqref{eq1.15} and \eqref{eq2.1} that the generalized Apostol-Bernoulli polynomials can be expressed in terms of the Apostol-Bernoulli polynomials:
\begin{equation}\label{eq2.5}
\beta_{n,\chi}(x,\lambda)=q^{n-1}\sum_{r=1}^{q}\chi(r)\lambda^{r}\beta_{n}\biggl(\frac{x+r}{q},\lambda^{q}\biggl)\quad(n\in\mathbb{N}_{0}).
\end{equation}
In addition, if we take $\lambda=1$ in \eqref{eq2.1}, then we have $\beta_{n,\chi}(x,1)=B_{n}(x,\chi)$ for $n\in\mathbb{N}_{0}$. If we take $\lambda=e^{\frac{\pi \mathrm{i}}{q}}$, then we have
\begin{equation}\label{eq2.6}
-\frac{2\beta_{n+1,\chi}(x,e^{\frac{\pi \mathrm{i}}{q}})}{n+1}=e_{n}(x,\chi)\quad(n\in\mathbb{N}_{0}),
\end{equation}
where $e_{n}(x,\chi)$ are called the generalized Euler polynomials of the first kind, defined for a Dirichlet character $\chi$ modulo $q\in\mathbb{N}$, by the generating function
\begin{equation}\label{eq2.7}
\sum_{r=1}^{q}\frac{\chi(r)e^{\frac{\pi \mathrm{i}r}{q}}2e^{(r+x)t}}{e^{qt}+1}=\sum_{n=0}^{\infty}e_{n}(x,\chi)\frac{t^{n}}{n!}\quad(|t|<\pi/q).
\end{equation}
If we take $\lambda=-1$, then we have
\begin{equation}\label{eq2.8}
-\frac{2\beta_{n+1,\chi}(x,-1)}{n+1}=E_{n}(x,\chi)\quad(n\in\mathbb{N}_{0},2\nmid q),
\end{equation}
where $E_{n}(x,\chi)$ are called the generalized Euler polynomials of the second kind, defined for a Dirichlet character $\chi$ modulo $q\in\mathbb{N}$, by the generating function
\begin{equation}\label{eq2.9}
\sum_{r=1}^{q}\frac{\chi(r)(-1)^{r}2e^{(r+x)t}}{e^{qt}+1}=\sum_{n=0}^{\infty}E_{n}(x,\chi)\frac{t^{n}}{n!}\quad(|t|<\pi/q).
\end{equation}

We now give some basic properties for the generalized Apostol-Bernoulli polynomials. The next results are analogous to the addition formulas and difference equations of the Bernoulli polynomials and the Euler polynomials stated in \cite[Chapter 2]{norlund}.

\begin{proposition}\label{pro2.2} Let $q\in\mathbb{N}$ and $\lambda\in\mathbb{C}\setminus\{0\}$. Let $\chi$ be a Dirichlet character modulo $q$. Then, for $n\in\mathbb{N}_{0}$,
\begin{equation}\label{eq2.10}
\beta_{n,\chi}(x+y,\lambda)=\sum_{k=0}^{n}\binom{n}{k}\beta_{k,\chi}(x,\lambda)y^{n-k},
\end{equation}
and for $n\in\mathbb{N}$,
\begin{equation}\label{eq2.11}
\lambda^{q}\beta_{n,\chi}(x+q,\lambda)-\beta_{n,\chi}(x,\lambda)=n\sum_{r=1}^{q}\chi(r)\lambda^{r}(x+r)^{n-1}.
\end{equation}
\end{proposition}

\begin{proof}
Observe that
\begin{equation*}
\sum_{r=1}^{q}\frac{\chi(r)\lambda^{r}te^{(r+x+y)t}}{\lambda^{q}e^{qt}-1}=e^{yt}\sum_{r=1}^{q}\frac{\chi(r)\lambda^{r}te^{(r+x)t}}{\lambda^{q}e^{qt}-1},
\end{equation*}
and
\begin{equation*}
\lambda^{q}\sum_{r=1}^{q}\frac{\chi(r)\lambda^{r}te^{(r+x+q)t}}{\lambda^{q}e^{qt}-1}-\sum_{r=1}^{q}\frac{\chi(r)\lambda^{r}te^{(r+x)t}}{\lambda^{q}e^{qt}-1}=t\sum_{r=1}^{q}\chi(r)\lambda^{r}e^{(r+x)t}.
\end{equation*}
Applying the Taylor series of $e^{t}$ and the Cauchy product to the above two identities, we obtain from \eqref{eq2.1} that
\begin{equation*}
\sum_{n=0}^{\infty}\beta_{n,\chi}(x+y,\lambda)\frac{t^{n}}{n!}=\sum_{n=0}^{\infty}\biggl(\sum_{k=0}^{n}\binom{n}{k}\beta_{k,\chi}(x,\lambda)y^{n-k}\biggl)\frac{t^{n}}{n!},
\end{equation*}
and
\begin{equation*}
\lambda^{q}\sum_{n=0}^{\infty}\beta_{n,\chi}(x+q,\lambda)\frac{t^{n}}{n!}-\sum_{n=0}^{\infty}\beta_{n,\chi}(x,\lambda)\frac{t^{n}}{n!}
=\sum_{n=0}^{\infty}\biggl(\sum_{r=1}^{q}\chi(r)\lambda^{r}(x+r)^{n}\biggl)\frac{t^{n+1}}{n!}.
\end{equation*}
Hence, the desired results follow when comparing the coefficients of $t^{n}/n!$.
\end{proof}

We also have the symmetric distributions of the generalized Apostol-Bernoulli polynomials.

\begin{proposition}\label{pro2.3} Let $q\in\mathbb{N}$, $n\in\mathbb{N}_{0}$ and $\lambda\in\mathbb{C}\setminus\{0\}$. Let $\chi$ be a Dirichlet character modulo $q$. If $q=1$, then we have
\begin{equation}\label{eq2.12}
\beta_{n,\chi}(-x,\lambda)=(-1)^{n}\beta_{n}(x,\lambda^{-1}),
\end{equation}
and if $q\geqslant2$, then we have
\begin{equation}\label{eq2.13}
\beta_{n,\chi}(-x,\lambda)=(-1)^{n}\beta_{n,\chi^{-}}(x,\lambda^{-1}),
\end{equation}
where $\chi^{-}$ denotes $\chi^{-}(r)=\chi(-r)$ for $r\in\mathbb{Z}$.
\end{proposition}

\begin{proof}
Clearly, the following identities hold:
\begin{eqnarray*}
\sum_{r=1}^{q}\frac{\chi(r)\lambda^{r}te^{(r-x)t}}{\lambda^{q}e^{qt}-1}
&=&\sum_{r=1}^{q}\frac{\chi(r)\lambda^{r}te^{(-r+x)(-t)}}{\lambda^{q}e^{qt}-1}\nonumber\\
&=&\sum_{r=0}^{q-1}\frac{\chi(q-r)\lambda^{q-r}te^{(-q+r+x)(-t)}}{\lambda^{q}e^{qt}-1}\nonumber\\
&=&\sum_{r=0}^{q-1}\frac{\chi(-r)\lambda^{-r}(-t)e^{(r+x)(-t)}}{\lambda^{-q}e^{q(-t)}-1}.
\end{eqnarray*}
So if $q=1$, then we know from \eqref{eq1.15} and \eqref{eq2.1} that \eqref{eq2.12} holds; if $q\geqslant2$, then we have
\begin{equation*}
\sum_{r=1}^{q}\frac{\chi(r)\lambda^{r}te^{(r-x)t}}{\lambda^{q}e^{qt}-1}
=\sum_{r=1}^{q}\frac{\chi(-r)\lambda^{-r}(-t)e^{(r+x)(-t)}}{\lambda^{-q}e^{q(-t)}-1}.
\end{equation*}
Therefore, applying \eqref{eq2.1} to both sides of the above identity and then comparing the coefficients of $t^{n}/n!$, we obtain \eqref{eq2.13}.
\end{proof}

It follows easily from \eqref{eq2.2} and \eqref{eq2.3} that \eqref{eq2.12} reduces, respectively, to the symmetric distributions of the Apostol-Bernoulli polynomials and that of the Apostol-Euler polynomials. That is, for $n\in\mathbb{N}_{0}$ and $\lambda\in\mathbb{C}\setminus\{0\}$, we have
\begin{equation}\label{eq2.14}
\lambda\beta_{n}(1-x,\lambda)=(-1)^{n}\beta_{n}(x,\lambda^{-1}),
\end{equation}
and
\begin{equation}\label{eq2.15}
\lambda\varepsilon_{n}(1-x,\lambda)=(-1)^{n}\varepsilon_{n}(x,\lambda^{-1}).
\end{equation}
Moreover, we know from \eqref{eq2.2}, \eqref{eq2.3}, \eqref{eq2.11}, \eqref{eq2.12} and \eqref{eq2.14} that the Apostol-Bernoulli polynomials and the Apostol-Euler polynomials satisfy the difference equations:
\begin{equation}\label{eq2.16}
\lambda\beta_{n}(x+1,\lambda)-\beta_{n}(x,\lambda)=nx^{n-1}\quad(n\in\mathbb{N},\lambda\in\mathbb{C}\setminus\{0\}),
\end{equation}
and
\begin{equation}\label{eq2.17}
\lambda\varepsilon_{n}(x+1,\lambda)+\varepsilon_{n}(x,\lambda)=2x^{n}\quad(n\in\mathbb{N}_{0},\lambda\in\mathbb{C}\setminus\{0\}).
\end{equation}
Similarly, using \eqref{eq2.11}, \eqref{eq2.12} and \eqref{eq2.13}, we can also obtain the symmetric distributions and difference equations for the generalized Bernoulli polynomials and for the generalized Euler polynomials of the first and second kinds.

The following symmetric identity is more general than the multiplication formulas of the Bernoulli polynomials and the Euler polynomials stated in \cite[Chapter 2]{norlund}.

\begin{proposition}\label{pro2.4} Let $q,a,b\in\mathbb{N}$ and $\lambda,\mu\in\mathbb{C}\setminus\{0\}$ with $\lambda^{a}=\mu^{b}$. Let $\chi$ be a Dirichlet character modulo $q$. Then, for $n\in\mathbb{N}_{0}$,
\begin{equation}\label{eq2.18}
a^{n-1}\sum_{k=1}^{aq}\chi(k)\lambda^{k}\beta_{n,\chi}\biggl(bx+\frac{bk}{a},\mu\biggl)=b^{n-1}\sum_{k=1}^{bq}\chi(k)\mu^{k}\beta_{n,\chi}\biggl(ax+\frac{ak}{b},\lambda\biggl).
\end{equation}
\end{proposition}

\begin{proof}
It is easily seen that
\begin{eqnarray}\label{eq2.19}
&&\frac{1}{a}\sum_{k=1}^{aq}\chi(k)\lambda^{k}\sum_{r=1}^{q}\frac{\chi(r)\mu^{r}ate^{(r+bx+\frac{bk}{a})at}}{\mu^{q}e^{qat}-1}\nonumber\\
&&\qquad=\frac{te^{abxt}}{\mu^{q}e^{qat}-1}\sum_{k=1}^{aq}\chi(k)\lambda^{k}e^{btk}\sum_{r=1}^{q}\chi(r)\mu^{r}e^{atr},
\end{eqnarray}
and
\begin{eqnarray}\label{eq2.20}
\sum_{k=1}^{aq}\chi(k)\lambda^{k}e^{btk}&=&\sum_{k=1}^{q}\chi(k)(\lambda e^{bt})^{k}+\sum_{k=1}^{q}\chi(q+k)(\lambda e^{bt})^{q+k}+\cdots\nonumber\\
&&+\sum_{k=1}^{q}\chi\bigl((a-1)q+k\bigl)(\lambda e^{bt})^{(a-1)q+k}\nonumber\\
&=&\bigl(1+\lambda^{q}e^{qbt}+\cdots+(\lambda^{q}e^{qbt})^{a-1}\bigl)\sum_{k=1}^{q}\chi(k)\lambda^{k}e^{btk}\nonumber\\
&=&\frac{\lambda^{aq}e^{abqt}-1}{\lambda^{q}e^{qbt}-1}\sum_{k=1}^{q}\chi(k)\lambda^{k}e^{btk}.
\end{eqnarray}
Hence, inserting \eqref{eq2.20} into \eqref{eq2.19}, we have
\begin{eqnarray}\label{eq2.21}
&&\frac{1}{a}\sum_{k=1}^{aq}\chi(k)\lambda^{k}\sum_{r=1}^{q}\frac{\chi(r)\mu^{r}ate^{(r+bx+\frac{bk}{a})at}}{\mu^{q}e^{qat}-1}\nonumber\\
&&\qquad=\frac{te^{abxt}(\lambda^{aq}e^{abqt}-1)}{(\mu^{q}e^{qat}-1)(\lambda^{q}e^{qbt}-1)}\sum_{k=1}^{q}\chi(k)\lambda^{k}e^{btk}\sum_{r=1}^{q}\chi(r)\mu^{r}e^{atr}.
\end{eqnarray}
After swapping $a$ and $b$, $\lambda$ and $\mu$ in \eqref{eq2.21}, we obtain
\begin{eqnarray}\label{eq2.22}
&&\frac{1}{b}\sum_{k=1}^{bq}\chi(k)\mu^{k}\sum_{r=1}^{q}\frac{\chi(r)\lambda^{r}bte^{(r+ax+\frac{ak}{b})bt}}{\lambda^{q}e^{qbt}-1}\nonumber\\
&&\qquad=\frac{te^{abxt}(\mu^{bq}e^{abqt}-1)}{(\lambda^{q}e^{qbt}-1)(\mu^{q}e^{qat}-1)}\sum_{k=1}^{q}\chi(k)\mu^{k}e^{atk}\sum_{r=1}^{q}\chi(r)\lambda^{r}e^{btr}.
\end{eqnarray}
Since $\lambda^{a}=\mu^{b}$, by \eqref{eq2.21} and \eqref{eq2.22} we have
\begin{eqnarray}\label{eq2.23}
&&\frac{1}{a}\sum_{k=1}^{aq}\chi(k)\lambda^{k}\sum_{r=1}^{q}\frac{\chi(r)\mu^{r}ate^{(r+bx+\frac{bk}{a})at}}{\mu^{q}e^{qat}-1}\nonumber\\
&&\qquad=\frac{1}{b}\sum_{k=1}^{bq}\chi(k)\mu^{k}\sum_{r=1}^{q}\frac{\chi(r)\lambda^{r}bte^{(r+ax+\frac{ak}{b})bt}}{\lambda^{q}e^{qbt}-1}.
\end{eqnarray}
Thus, by applying \eqref{eq2.1} to both sides of \eqref{eq2.23} and then comparing the coefficients of $t^{n}/n!$, we get \eqref{eq2.18} and finish the proof of Proposition \ref{pro2.4}.
\end{proof}

As applications of Proposition \ref{pro2.4}, we have the following results.

\begin{corollary}\label{cor2.5} Let $a,b\in\mathbb{N}$ and $\lambda,\mu\in\mathbb{C}\setminus\{0\}$ with $\lambda^{a}=\mu^{b}$. Then, for $n\in\mathbb{N}_{0}$,
\begin{equation}\label{eq2.24}
a^{n-1}\sum_{k=0}^{a-1}\lambda^{k}\beta_{n}\biggl(bx+\frac{bk}{a},\mu\biggl)=b^{n-1}\sum_{k=0}^{b-1}\mu^{k}\beta_{n}\biggl(ax+\frac{ak}{b},\lambda\biggl),
\end{equation}
and
\begin{equation}\label{eq2.25}
a^{n}\sum_{k=0}^{a-1}\lambda^{k}\varepsilon_{n}\biggl(bx+\frac{bk}{a},-\mu\biggl)=b^{n}\sum_{k=0}^{b-1}\mu^{k}\varepsilon_{n}\biggl(ax+\frac{ak}{b},-\lambda\biggl).
\end{equation}
\end{corollary}

\begin{proof}
Setting $q=1$ in \eqref{eq2.18} and noting that $\lambda^{a}=\mu^{b}$, we obtain
\begin{equation}\label{eq2.26}
a^{n-1}\sum_{k=0}^{a-1}\lambda^{-k}\beta_{n,\chi}\biggl(bx+b-\frac{bk}{a},\mu\biggl)=b^{n-1}\sum_{k=0}^{b-1}\mu^{-k}\beta_{n,\chi}\biggl(ax+a-\frac{ak}{b},\lambda\biggl).
\end{equation}
Hence, from \eqref{eq2.12} and \eqref{eq2.26}, it follows that
\begin{equation*}
a^{n-1}\sum_{k=0}^{a-1}\lambda^{-k}\beta_{n}\biggl(-bx-b+\frac{bk}{a},\mu^{-1}\biggl)=b^{n-1}\sum_{k=0}^{b-1}\mu^{-k}\beta_{n}\biggl(-ax-a+\frac{ak}{b},\lambda^{-1}\biggl).
\end{equation*}
Replacing $\lambda$ by $\lambda^{-1}$, $\mu$ by $\mu^{-1}$ and $x$ by $-x-1$ in the above identity, we get \eqref{eq2.24}. If we replace $n$ by $n+1$ in \eqref{eq2.26} and use \eqref{eq2.3}, then we have
\begin{equation*}
-\mu a^{n}\sum_{k=0}^{a-1}\lambda^{-k}\varepsilon_{n}\biggl(bx+b+1-\frac{bk}{a},-\mu\biggl)=-\lambda b^{n}\sum_{k=0}^{b-1}\mu^{-k}\varepsilon_{n}\biggl(ax+a+1-\frac{ak}{b},-\lambda\biggl).
\end{equation*}
It follows from \eqref{eq2.15} that
\begin{equation*}
a^{n}\sum_{k=0}^{a-1}\lambda^{-k}\varepsilon_{n}\biggl(-bx-b+\frac{bk}{a},-\mu^{-1}\biggl)=b^{n}\sum_{k=0}^{b-1}\mu^{-k}\varepsilon_{n}\biggl(-ax-a+\frac{ak}{b},-\lambda^{-1}\biggl).
\end{equation*}
Thus, \eqref{eq2.25} follows by replacing $\lambda,\mu,x$ with $\lambda^{-1}, \mu^{-1},-x-1$, respectively.
\end{proof}

Corollary \ref{cor2.5} above was also obtained earlier by Luo \cite{luo1} in 2009. In particular, setting $\lambda=\mu=b=1$ in \eqref{eq2.24} yields that for $n\in\mathbb{N}_{0}$ and $a\in\mathbb{N}$,
\begin{equation}\label{eq2.27}
a^{n-1}\sum_{k=0}^{a-1}B_{n}\biggl(x+\frac{k}{a}\biggl)=B_{n}(ax),
\end{equation}
and setting $\lambda=\mu=-1$ and $b=1$ in \eqref{eq2.25} yields that for $n\in\mathbb{N}_{0}$ and $a\in\mathbb{N}$ with $2\nmid a$,
\begin{equation}\label{eq2.28}
a^{n}\sum_{k=0}^{a-1}(-1)^{k}E_{n}\biggl(x+\frac{k}{a}\biggl)=E_{n}(ax).
\end{equation}
In addition, setting $\lambda=-1$ and $\mu=b=1$ in \eqref{eq2.24} also yields that for $n,a\in\mathbb{N}$ with $2\mid a$,
\begin{equation}\label{eq2.29}
a^{n-1}\sum_{k=0}^{a-1}(-1)^{k}B_{n}\biggl(x+\frac{k}{a}\biggl)=-\frac{nE_{n-1}(ax)}{2}.
\end{equation}
It is worth mentioning that formula \eqref{eq2.27} is attributed to Raabe \cite{raabe}, and is usually called Raabe's multiplication formula. In particular, Howard \cite{howard} in 1995 used Raabe's multiplication formula \eqref{eq2.27} to prove the well-known Staudt-Clausen theorem and the theorems of Carlitz, Frobenius, and Ramanujan.

\begin{corollary}\label{cor2.6} Let $q,a,b\in\mathbb{N}$ and $n\in\mathbb{N}_{0}$. Let $\chi$ be a Dirichlet character modulo $q$. Then
\begin{equation}\label{eq2.30}
a^{n-1}\sum_{k=1}^{aq}\chi(k)B_{n}\biggl(bx+\frac{bk}{a},\chi\biggl)=b^{n-1}\sum_{k=1}^{bq}\chi(k)B_{n}\biggl(ax+\frac{ak}{b},\chi\biggl).
\end{equation}
Furthermore, if $e^{\frac{\pi \mathrm{i}a}{q}}=e^{\frac{\pi \mathrm{i}b}{q}}$, then
\begin{equation}\label{eq2.31}
a^{n}\sum_{k=1}^{aq}\chi(k)e^{\frac{\pi \mathrm{i}k}{q}}e_{n}\biggl(bx+\frac{bk}{a},\chi\biggl)=b^{n}\sum_{k=1}^{bq}\chi(k)e^{\frac{\pi \mathrm{i}k}{q}}e_{n}\biggl(ax+\frac{ak}{b},\chi\biggl),
\end{equation}
and if $(-1)^{a}=(-1)^{b}$ and $2\nmid q$, then
\begin{equation}\label{eq2.32}
a^{n}\sum_{k=1}^{aq}\chi(k)(-1)^{k}E_{n}\biggl(bx+\frac{bk}{a},\chi\biggl)=b^{n}\sum_{k=1}^{bq}\chi(k)(-1)^{k}E_{n}\biggl(ax+\frac{ak}{b},\chi\biggl).
\end{equation}
\end{corollary}

\begin{proof}
Taking $\lambda=\mu=1$ in \eqref{eq2.18} gives \eqref{eq2.30}. Taking $\lambda=\mu=e^{\frac{\pi \mathrm{i}}{q}}$ in \eqref{eq2.18} and using \eqref{eq2.6} gives \eqref{eq2.31}. Taking $\lambda=\mu=-1$ in \eqref{eq2.18} and using \eqref{eq2.8} gives \eqref{eq2.32}.
\end{proof}

We next introduce the generalized Apostol-Bernoulli functions.

\begin{definition}\label{def2.7} Let $n\in\mathbb{N}_{0}$, $x\in\mathbb{R}$ and $\lambda\in\mathbb{C}\setminus\{0\}$. For a Dirichlet character $\chi$ modulo $q\in\mathbb{N}$, the generalized Apostol-Bernoulli functions $\overline{\beta}_{n,\chi}(x,\lambda)$ are defined by
\begin{equation}\label{eq2.33}
\overline{\beta}_{n,\chi}(x,\lambda)=
\beta_{n,\chi}(\{x\}_{\chi},\lambda)+\frac{1}{2}\delta_{1,n}\delta_{\mathbb{Z}}(x)\lambda^{-x}\chi(-x),
\end{equation}
where $\delta_{l,k}$ is the Kronecker delta-function, defined for $l,k\in\mathbb{N}_{0}$ by $\delta_{l,k}=1$ if $l=k$ and $0$ otherwise, and $\delta_{\mathbb{Z}}(x)=1$ if $x\in\mathbb{Z}$ and $0$ otherwise; and
\begin{equation}\label{eq2.34}
\beta_{n,\chi}(\{x\}_{\chi},\lambda)=q^{n-1}\sum_{k=1}^{q}\chi(k)\lambda^{k-q[\frac{x+k}{q}]}\beta_{n}\biggl(\biggl\{\frac{x+k}{q}\biggl\},\lambda^{q}\biggl),
\end{equation}
where $\{x\}$ denotes the fractional part of $x$ and $[x]$ denotes the integer part satisfying $\{x\}+[x]=x$.
\end{definition}

In particular, consider the case
$q=1$ in Definition \ref{def2.7}. This yields, respectively, the Apostol-Bernoulli functions $\overline{\beta}_{n}(x,\lambda)$ and the Apostol-Euler functions $\overline{\varepsilon}_{n}(x,\lambda)$, which are defined for $x\in\mathbb{R}$ and $\lambda\in\mathbb{C}\setminus\{0\}$ by
\begin{equation}\label{eq2.35}
\overline{\beta}_{n}(x,\lambda)=\lambda^{-[x]}\beta_{n}(\{x\},\lambda)+\frac{1}{2}\delta_{1,n}\delta_{\mathbb{Z}}(x)\lambda^{-x}\quad(n\in\mathbb{N}_{0}),
\end{equation}
and
\begin{equation}\label{eq2.36}
\overline{\varepsilon}_{n}(x,\lambda)=(-1)^{[x]}\lambda^{-[x]}\varepsilon_{n}(\{x\},\lambda)-\delta_{0,n}\delta_{\mathbb{Z}}(x)(-\lambda)^{-x}\quad(n\in\mathbb{N}_{0}).
\end{equation}
We also call the functions $\overline{B}_{n}(x)=\overline{\beta}_{n}(x,1)$ the Bernoulli functions, and $\overline{E}_{n}(x)=\overline{\varepsilon}_{n}(x,1)$ the Euler functions.

On the other hand, if we set $\lambda=1$, $\lambda=e^{\frac{\pi \mathrm{i}}{q}}$, and $\lambda=-1$ separately in Definition \ref{def2.7}, we obtain respectively the generalized Bernoulli functions $\overline{B}_{n}(x,\chi)$, the generalized Euler functions of the first kind $\overline{e}_{n}(x,\chi)$, and the generalized Euler functions of the second kind $\overline{E}_{n}(x,\chi)$, which are defined for $x\in\mathbb{R}$ and for a Dirichlet character $\chi$ modulo $q\in\mathbb{N}$ by
\begin{equation}\label{eq2.37}
\overline{B}_{n}(x,\chi)=B_{n}(\{x\}_{\chi},\chi)+\frac{1}{2}\delta_{1,n}\delta_{\mathbb{Z}}(x)\chi(-x)\quad(n\in\mathbb{N}_{0}),
\end{equation}
\begin{equation}\label{eq2.38}
\overline{e}_{n}(x,\chi)=e_{n}(\{x\},\chi)-\delta_{0,n}\delta_{\mathbb{Z}}(x)e^{-\frac{\pi \mathrm{i}x}{q}}\chi(-x)\quad(n\in\mathbb{N}_{0}),
\end{equation}
and
\begin{equation}\label{eq2.39}
\overline{E}_{n}(x,\chi)=E_{n}(\{x\},\chi)-\delta_{0,n}\delta_{\mathbb{Z}}(x)(-1)^{x}\chi(-x)\quad(n\in\mathbb{N}_{0},2\nmid q),
\end{equation}
where
\begin{equation*}
B_{n}(\{x\}_{\chi},\chi)=q^{n-1}\sum_{k=1}^{q}\chi(k)B_{n}\biggl(\biggl\{\frac{x+k}{q}\biggl\}\biggl),
\end{equation*}
\begin{equation*}
e_{n}(\{x\}_{\chi},\chi)=q^{n}\sum_{k=1}^{q}\chi(k)e^{\frac{\pi \mathrm{i}k}{q}-\pi\mathrm{i}[\frac{x+k}{q}]}E_{n}\biggl(\biggl\{\frac{x+k}{q}\biggl\}\biggl),
\end{equation*}
and
\begin{equation*}
E_{n}(\{x\}_{\chi},\chi)=q^{n}\sum_{k=1}^{q}\chi(k)(-1)^{k-q[\frac{x+k}{q}]}E_{n}\biggl(\biggl\{\frac{x+k}{q}\biggl\}\biggl).
\end{equation*}
Trivially, setting $q=1$ in \eqref{eq2.37} gives the Bernoulli functions, while the same substitution in either \eqref{eq2.38} or \eqref{eq2.39} yields the Euler functions.

It is easily seen that $\lambda^{x}\beta_{n,\chi}(\{x\}_{\chi},\lambda)$ is a periodic function on $\mathbb{R}$ of period $q$. It is continuous for
$n\geqslant2$, but for $n=1$ it has simple discontinuities at integers, satisfying that for $m\in\mathbb{Z}$ with $\gcd(m,q)=1$,
\begin{equation*}
\underset{\substack{x\rightarrow m\\ x>m}}{\beta_{1,\chi}(\{x\}_{\chi},\lambda)}=\beta_{1,\chi}(\{m\}_{\chi},\lambda),
\end{equation*}
and
\begin{equation*}
\underset{\substack{x\rightarrow m\\ x<m}}{\beta_{1,\chi}(\{x\}_{\chi},\lambda)}=\beta_{1,\chi}(\{m\}_{\chi},\lambda)+\lambda^{-m}\chi(-m),
\end{equation*}
which can be obtained by using \eqref{eq2.16} and the division algorithm stated in \cite[Theorem 1.14]{apostol4}.

It is interesting to point out that there exists a connection between $\beta_{n,\chi}(\{x\}_{\chi},\lambda)$ and $\beta_{n,\chi}(x,\lambda)$ as follows.

\begin{proposition}\label{pro2.8} Let $q,n\in\mathbb{N}$, $x\in\mathbb{R}$ and $\lambda\in\mathbb{C}\setminus\{0\}$. Let $\chi$ be a Dirichlet character modulo $q$. Then
\begin{equation}\label{eq2.40}
\beta_{n,\chi}(\{x\}_{\chi},\lambda)=\beta_{n,\chi}(x,\lambda)-n\underset{\substack{0\leqslant k\leqslant x\\ k\in\mathbb{Z}}}{\sum}\chi(-k)\lambda^{-k}(x-k)^{n-1}.
\end{equation}
\end{proposition}

\begin{proof}
Let
\begin{equation*}
C_{n}(x)=\underset{\substack{0\leqslant k\leqslant x\\ k\in\mathbb{Z}}}{\sum}\chi(-k)\lambda^{-k}(x-k)^{n-1}.
\end{equation*}
Then, we have
\begin{eqnarray}\label{eq2.41}
C_{n}(x+q)&=&\underset{\substack{0\leqslant k\leqslant q-1\\ k\in\mathbb{Z}}}{\sum}\chi(-k)\lambda^{-k}(x+q-k)^{n-1}\nonumber\\
&&+\underset{\substack{q\leqslant k\leqslant x+q\\ k\in\mathbb{Z}}}{\sum}\chi(-k)\lambda^{-k}(x+q-k)^{n-1}\nonumber\\
&=&\lambda^{-q}\underset{\substack{0\leqslant k\leqslant q-1\\ k\in\mathbb{Z}}}{\sum}\chi(q-k)\lambda^{q-k}(x+q-k)^{n-1}\nonumber\\
&&+\lambda^{-q}\underset{\substack{0\leqslant k\leqslant x\\ k\in\mathbb{Z}}}{\sum}\chi(-k)\lambda^{-k}(x-k)^{n-1}\nonumber\\
&=&\lambda^{-q}\underset{\substack{1\leqslant k\leqslant q\\ k\in\mathbb{Z}}}{\sum}\chi(k)\lambda^{k}(x+k)^{n-1}+\lambda^{-q}C_{n}(x).
\end{eqnarray}
Applying \eqref{eq2.11} to the right-hand side of \eqref{eq2.41}, we obtain
\begin{equation*}
C_{n}(x+q)-\lambda^{-q}C_{n}(x)=\frac{\beta_{n,\chi}(x+q,\lambda)-\lambda^{-q}\beta_{n,\chi}(x,\lambda)}{n},
\end{equation*}
which means
\begin{equation*}
\lambda^{x+q}\bigl(\beta_{n,\chi}(x+q,\lambda)-nC_{n}(x+q)\bigl)=\lambda^{x}\bigl(\beta_{n,\chi}(x,\lambda)-nC_{n}(x)\bigl).
\end{equation*}
Hence, $\lambda^{x}\bigl(\beta_{n,\chi}(x,\lambda)-nC_{n}(x)\bigl)$ is a periodic function on $\mathbb{R}$ of period $q$.
Since $\lambda^{x}\beta_{n,\chi}(\{x\}_{\chi},\lambda)$ is also such a function, it suffices to prove that \eqref{eq2.40} holds in the case when $0\leqslant x<q$. Obviously, in this case, we have
\begin{eqnarray*}
\beta_{n,\chi}(\{x\}_{\chi},\lambda)&=&q^{n-1}\underset{\substack{1\leqslant k< q-x\\ k\in\mathbb{Z}}}{\sum}\chi(k)\lambda^{k-q[\frac{x+k}{q}]}\beta_{n}\biggl(\biggl\{\frac{x+k}{q}\biggl\},\lambda^{q}\biggl)\nonumber\\
&&+q^{n-1}\underset{\substack{q-x\leqslant k\leqslant q\\ k\in\mathbb{Z}}}{\sum}\chi(k)\lambda^{k-q[\frac{x+k}{q}]}\beta_{n}\biggl(\biggl\{\frac{x+k}{q}\biggl\},\lambda^{q}\biggl)\nonumber\\
&=&q^{n-1}\underset{\substack{1\leqslant k< q-x\\ k\in\mathbb{Z}}}{\sum}\chi(k)\lambda^{k}\beta_{n}\biggl(\frac{x+k}{q},\lambda^{q}\biggl)\nonumber\\
&&+q^{n-1}\lambda^{-q}\underset{\substack{q-x\leqslant k\leqslant q\\ k\in\mathbb{Z}}}{\sum}\chi(k)\lambda^{k}\beta_{n}\biggl(\frac{x+k}{q}-1,\lambda^{q}\biggl).
\end{eqnarray*}
Thus, we know from \eqref{eq2.5} and \eqref{eq2.16} that the above identity can be rewritten as
\begin{eqnarray*}
\beta_{n,\chi}(\{x\}_{\chi},\lambda)&=&q^{n-1}\underset{\substack{1\leqslant k< q-x\\ k\in\mathbb{Z}}}{\sum}\chi(k)\lambda^{k}\beta_{n}\biggl(\frac{x+k}{q},\lambda^{q}\biggl)\nonumber\\
&&+q^{n-1}\underset{\substack{q-x\leqslant k\leqslant q\\ k\in\mathbb{Z}}}{\sum}\chi(k)\lambda^{k}\beta_{n}\biggl(\frac{x+k}{q},\lambda^{q}\biggl)\nonumber\\
&&-n\lambda^{-q}\underset{\substack{q-x\leqslant k\leqslant q\\ k\in\mathbb{Z}}}{\sum}\chi(k)\lambda^{k}(x+k-q)^{n-1}\nonumber\\
&=&\beta_{n,\chi}(x,\lambda)-n\lambda^{-q}\underset{\substack{q-x\leqslant k\leqslant q\\ k\in\mathbb{Z}}}{\sum}\chi(k)\lambda^{k}(x+k-q)^{n-1}\nonumber\\
&=&\beta_{n,\chi}(x,\lambda)-n\underset{\substack{0\leqslant k\leqslant x\\ k\in\mathbb{Z}}}{\sum}\chi(-k)\lambda^{-k}(x-k)^{n-1}.
\end{eqnarray*}
This completes the proof of Proposition \ref{pro2.8}.
\end{proof}

\section{Fourier series of generalized Apostol-Bernoulli functions}\label{sec3}

We now present the Fourier series of the generalized Apostol-Bernoulli functions, which is given as follows.

\begin{theorem}\label{thm3.1} Let $n,q\in\mathbb{N}$, $x\in\mathbb{R}$ and $\lambda\in\mathbb{C}\setminus\{0\}$. Let $\chi$ be a Dirichlet character modulo $q$. Then
\begin{equation}\label{eq3.1}
\overline{\beta}_{n,\chi}(x,\lambda)=-\frac{q^{n-1}n!}{\lambda^{x}(2\pi \mathrm{i})^{n}}\sideset{}{'}\sum_{k=-\infty}^{+\infty}\frac{G(k,\chi)e^{\frac{2\pi\mathrm{i}kx}{q}}}{(k-\frac{q\log\lambda}{2\pi\mathrm{i}})^{n}},
\end{equation}
where the dash denotes throughout that undefined terms are excluded from the
sum, and $G(k,\chi)$ is the Gauss sum defined as in \eqref{eq1.11}.
\end{theorem}

\begin{proof}
Since $\lambda^{x}\beta_{n,\chi}(\{x\}_{\chi},\lambda)$ is of bounded variation on every finite interval, by Dirichlet-Jordan test we have
\begin{equation}\label{eq3.2}
\lambda^{x}\overline{\beta}_{n,\chi}(x,\lambda)=\sum_{k=-\infty}^{+\infty}c_{n,k}(q)e^{\frac{2\pi\mathrm{i}kx}{q}},
\end{equation}
where the Fourier coefficients $c_{n,k}(q)$ are determined by
\begin{eqnarray}\label{eq3.3}
c_{n,k}(q)&=&\frac{1}{q}\int_{0}^{q}\lambda^{t}\beta_{n,\chi}(\{t\}_{\chi},\lambda)e^{-\frac{2\pi\mathrm{i}kt}{q}}\text{d}t\nonumber\\
&=&\frac{1}{q}\int_{0}^{q}\beta_{n,\chi}(\{t\}_{\chi},\lambda)e^{(\log\lambda-\frac{2\pi\mathrm{i}k}{q})t}\text{d}t.
\end{eqnarray}
Trivially, for $q=1$, we have
\begin{eqnarray}\label{eq3.4}
c_{n,k}(1)&=&\int_{0}^{1}\beta_{n}(t,\lambda)e^{(\log\lambda-2\pi\mathrm{i}k)t}\text{d}t.
\end{eqnarray}
Note that by taking the derivative with respect to $x$ on both sides of \eqref{eq1.15}, we have
\begin{equation}\label{eq3.5}
\frac{\partial}{\partial x}\beta_{n}(x,\lambda)=n\beta_{n-1}(x,\lambda)\quad(n\in\mathbb{N}).
\end{equation}
Hence, using integration by parts on \eqref{eq3.4} and in light of \eqref{eq2.16} and \eqref{eq3.5}, we know that if $\lambda\neq1$, or if $\lambda=1$ and $k\neq0$, then
\begin{eqnarray}\label{eq3.6}
c_{n,k}(1)&=&-\frac{1}{2\pi\mathrm{i}k-\log\lambda}\bigl(\lambda\beta_{n}(1,\lambda)-\beta_{n}(0,\lambda)-nc_{n-1,k}(1)\bigl)\nonumber\\
&=&-\frac{1}{2\pi\mathrm{i}k-\log\lambda}\bigl(\delta_{1,n}-nc_{n-1,k}(1)\bigl).
\end{eqnarray}
Also, from \eqref{eq1.15}, we have
\begin{equation*}
(\lambda e^{t}-1)\sum_{n=0}^{\infty}\beta_{n}(x,\lambda)\frac{t^{n}}{n!}=te^{xt}.
\end{equation*}
Applying the Taylor series of $e^{t}$ to both sides of the above identity, we find that
\begin{equation}\label{eq3.7}
\beta_{0}(x,\lambda)=\begin{cases}
1,  &\lambda=1,\\
0,  &\lambda\not=1,
\end{cases}
\end{equation}
and
\begin{equation}\label{eq3.8}
\beta_{1}(x,\lambda)=\begin{cases}
x-\frac{1}{2},  &\lambda=1,\\
\frac{1}{\lambda-1},  &\lambda\not=1.
\end{cases}
\end{equation}
It follows from \eqref{eq3.7} that if $\lambda\neq1$, or if $\lambda=1$ and $k\neq0$, then
\begin{equation*}
c_{0,k}(1)=0,
\end{equation*}
which together with \eqref{eq3.6} yields, for $\lambda\neq1$ or ($\lambda=1$ and $k\neq0$),
\begin{equation}\label{eq3.9}
c_{n,k}(1)=-\frac{n!}{(2\pi\mathrm{i}k-\log\lambda)^{n}}.
\end{equation}
Because for $\lambda=1$, by \eqref{eq2.16} and \eqref{eq3.5} we have
\begin{equation}\label{eq3.10}
c_{n,0}(1)=\int_{0}^{1}\beta_{n}(t,1)\text{d}t=0\quad(n\in\mathbb{N}),
\end{equation}
we thus see, by inserting \eqref{eq3.9} and \eqref{eq3.10} into \eqref{eq3.2}, that
\begin{equation}\label{eq3.11}
\overline{\beta}_{n}(x,\lambda)=-\frac{n!}{\lambda^{x}(2\pi \mathrm{i})^{n}}\sideset{}{'}\sum_{k=-\infty}^{+\infty}\frac{e^{2\pi\mathrm{i}kx}}{(k-\frac{\log\lambda}{2\pi\mathrm{i}})^{n}},
\end{equation}
which means Theorem \ref{thm3.1} holds for $q=1$.

We next consider the case $q\geqslant2$. Taking the derivative with respect to $x$ on both sides of \eqref{eq3.11}, we obtain that for $n\geqslant2$ and $x\in\mathbb{R}$,
\begin{equation*}
\frac{\partial}{\partial x}\overline{\beta}_{n}(x,\lambda)=n\overline{\beta}_{n-1}(x,\lambda).
\end{equation*}
From this, \eqref{eq2.34} and \eqref{eq2.35}, we see that
\begin{equation}\label{eq3.12}
\frac{\partial}{\partial x}\beta_{n,\chi}(\{x\}_{\chi},\lambda)=n\beta_{n-1,\chi}(\{x\}_{\chi},\lambda),
\end{equation}
for $n\geqslant2$ and $x\in\mathbb{R}$, except when $n=2$ and $x$ belongs to a certain set of measure zero (namely, the set of $x$ for which there exists $k\in\{1,2,\ldots,q\}$ such that $q\mid(x+k)$). Using integration by parts in \eqref{eq3.3} and taking \eqref{eq3.12} into account, we know that if $2\pi\mathrm{i}k-q\log\lambda\neq0$, then
\begin{eqnarray}\label{eq3.13}
c_{n,k}(q)=\frac{qn}{2\pi\mathrm{i}k-q\log\lambda}c_{n-1,k}(q)\quad(n\geqslant2).
\end{eqnarray}
Clearly, from \eqref{eq2.5}, \eqref{eq2.40} and \eqref{eq3.8}, it follows that if $\lambda^{q}=1$, then for $t\in\mathbb{R}$,
\begin{equation}\label{eq3.14}
\beta_{1,\chi}(\{t\}_{\chi},\lambda)=\frac{t}{q}\sum_{r=1}^{q}\chi(r)\lambda^{r}+\sum_{r=1}^{q}\chi(r)\lambda^{r}B_{1}\biggl(\frac{r}{q}\biggl)-\underset{\substack{0\leqslant r\leqslant t\\ r\in\mathbb{Z}}}{\sum}\chi(-r)\lambda^{-r},
\end{equation}
and if $\lambda^{q}\neq1$, then for $t\in\mathbb{R}$,
\begin{equation}\label{eq3.15}
\beta_{1,\chi}(\{t\}_{\chi},\lambda)=\frac{1}{\lambda^{q}-1}\sum_{r=1}^{q}\chi(r)\lambda^{r}-\underset{\substack{0\leqslant r\leqslant t\\ r\in\mathbb{Z}}}{\sum}\chi(-r)\lambda^{-r}.
\end{equation}
Inserting \eqref{eq3.14} and \eqref{eq3.15} into \eqref{eq3.3}, respectively, and then using integration by parts, we find that for $2\pi\mathrm{i}k-q\log\lambda\neq0$,
\begin{eqnarray}\label{eq3.16}
c_{1,k}&=&-\frac{1}{2\pi\mathrm{i}k-q\log\lambda}\sum_{r=1}^{q}\chi(r)\lambda^{r}\nonumber\\
&&-\frac{1}{q}\int_{0}^{q}\underset{\substack{0\leqslant r\leqslant t\\ r\in\mathbb{Z}}}{\sum}\chi(-r)\lambda^{-r}e^{(\log\lambda-\frac{2\pi\mathrm{i}k}{q})t}\text{d}t.
\end{eqnarray}
Observe that for $2\pi\mathrm{i}k-q\log\lambda\neq0$,
\begin{eqnarray}\label{eq3.17}
&&\int_{0}^{q}\underset{\substack{0\leqslant r\leqslant t\\ r\in\mathbb{Z}}}{\sum}\chi(-r)\lambda^{-r}e^{(\log\lambda-\frac{2\pi\mathrm{i}k}{q})t}\text{d}t\nonumber\\
&&\qquad=\chi(0)\lambda^{0}\int_{0}^{1}e^{(\log\lambda-\frac{2\pi\mathrm{i}k}{q})t}\text{d}t\nonumber\\
&&\qquad\quad+\bigl(\chi(0)\lambda^{0}+\chi(-1)\lambda^{-1}\bigl)\int_{1}^{2}e^{(\log\lambda-\frac{2\pi\mathrm{i}k}{q})t}\text{d}t\nonumber\\
&&\qquad\quad+\bigl(\chi(0)\lambda^{0}+\chi(-1)\lambda^{-1}+\chi(-2)\lambda^{-2}\bigl)\int_{2}^{3}e^{(\log\lambda-\frac{2\pi\mathrm{i}k}{q})t}\text{d}t\nonumber\\
&&\qquad\quad+\cdots\nonumber\\
&&\qquad\quad+\bigl(\chi(0)\lambda^{0}+\chi(-1)\lambda^{-1}+\cdots+\chi\bigl(-(q-1)\bigl)\lambda^{-(q-1)}\bigl)\int_{q-1}^{q}e^{(\log\lambda-\frac{2\pi\mathrm{i}k}{q})t}\text{d}t\nonumber\\
&&\qquad=\sum_{r=0}^{q-1}\chi(-r)\lambda^{-r}\int_{r}^{q}e^{(\log\lambda-\frac{2\pi\mathrm{i}k}{q})t}\text{d}t\nonumber\\
&&\qquad=-\frac{q}{2\pi\mathrm{i}k-q\log\lambda}\sum_{r=0}^{q-1}\chi(-r)\lambda^{-r}\bigl(\lambda^{q}-\lambda^{r}e^{-\frac{2\pi\mathrm{i}kr}{q}}\bigl)\nonumber\\
&&\qquad=-\frac{q}{2\pi\mathrm{i}k-q\log\lambda}\biggl(\sum_{r=1}^{q}\chi(r)\lambda^{r}-\sum_{r=1}^{q}\chi(r)e^{\frac{2\pi\mathrm{i}kr}{q}}\biggl).
\end{eqnarray}
So from \eqref{eq3.16} and \eqref{eq3.17}, we obtain that for $2\pi\mathrm{i}k-q\log\lambda\neq0$,
\begin{equation}\label{eq3.18}
c_{1,k}=-\frac{G(k,\chi)}{2\pi\mathrm{i}k-q\log\lambda}.
\end{equation}
It follows from \eqref{eq3.13} and \eqref{eq3.18} that if $2\pi\mathrm{i}k-q\log\lambda\neq0$, then for $n\in\mathbb{N}$,
\begin{eqnarray}\label{eq3.19}
c_{n,k}(q)=-\frac{q^{n-1}n!G(k,\chi)}{(2\pi\mathrm{i}k-q\log\lambda)^{n}}.
\end{eqnarray}
On the other hand, if $2\pi\mathrm{i}k-q\log\lambda=0$, then by using integration by parts, we know from \eqref{eq2.34}, \eqref{eq3.3} and \eqref{eq3.12} that for $n\in\mathbb{N}$,
\begin{eqnarray}\label{eq3.20}
c_{n,k}(q)&=&\frac{1}{q}\int_{0}^{q}\beta_{n,\chi}(\{t\}_{\chi},\lambda)\text{d}t\nonumber\\
&=&\frac{\beta_{n+1,\chi}(\{q\}_{\chi},\lambda)-\beta_{n+1,\chi}(\{0\}_{\chi},\lambda)}{q(n+1)}\nonumber\\
&=&0.
\end{eqnarray}
Therefore, inserting \eqref{eq3.19} and \eqref{eq3.20} into \eqref{eq3.2} yields \eqref{eq3.1}, which finishes the proof of Theorem \ref{thm3.1}.
\end{proof}

It follows that some special cases of Theorem \ref{thm3.1} can be obtained, namely the following Fourier series of the Apostol-Bernoulli functions and of the Apostol-Euler functions.

\begin{corollary}\label{cor3.2} Let $x\in\mathbb{R}$ and $\lambda\in\mathbb{C}\setminus\{0\}$. Then, for $n\in\mathbb{N}$,
\begin{equation}\label{eq3.21}
\overline{\beta}_{n}(x,\lambda)=-\frac{n!}{\lambda^{x}(2\pi \mathrm{i})^{n}}\sum_{k\in\mathbb{Z}}^{*}\frac{e^{2\pi\mathrm{i}kx}}{(k-\frac{\log\lambda}{2\pi\mathrm{i}})^{n}},
\end{equation}
and for $n\in\mathbb{N}_{0}$,
\begin{equation}\label{eq3.22}
\overline{\varepsilon}_{n}(x,\lambda)=\frac{2n!}{\lambda^{x}(2\pi \mathrm{i})^{n+1}}\sum_{k\in\mathbb{Z}}^{**}\frac{e^{2\pi\mathrm{i}(k-\frac{1}{2})x}}{(k-\frac{1}{2}-\frac{\log\lambda}{2\pi\mathrm{i}})^{n+1}},
\end{equation}
where the notation $\sum_{k\in\mathbb{Z}}^{*}$ means summation over $k\in\mathbb{Z}\setminus\{0\}$ if $\lambda=1$, and over $k\in\mathbb{Z}$ if $\lambda\neq1$; similarly, $\sum_{k\in\mathbb{Z}}^{**}$ means summation over $k\in\mathbb{Z}\setminus\{1\}$ if $\lambda=-1$, and over $k\in\mathbb{Z}$ if $\lambda\neq-1$.
\end{corollary}

\begin{proof}
Setting $q=1$ in \eqref{eq3.1} gives \eqref{eq3.21}. Similarly, setting $q=1$ in \eqref{eq3.1}, replacing $\lambda$ by $-\lambda$ and $n$ by $n+1$, and then multiplying by
$-2/(n+1)$ yields \eqref{eq3.22}.
\end{proof}

We remark that, for \eqref{eq3.21}, the cases $0<x<1$, $n=1$ and $0\leqslant x\leqslant 1$, $n\geqslant2$; and for \eqref{eq3.22}, the cases $0<x<1$, $n=0$ and $0\leqslant x\leqslant 1$, $n\geqslant1$, are in agreement with the results of Luo \cite{luo1}, who, in 2009, used the Lipschitz summation formula \cite{lipschitz} to prove them. Later, in 2011, Bayad \cite{bayad2} used contour integration to provide a new proof of Luo's results. It should be noted that Raabe \cite{raabe} in 1851 obtained the Fourier series of the Bernoulli polynomials. For the character analogue of the Lipschitz summation formula, see the 1975 work of Berndt \cite{berndt3}.

We also obtain the Fourier series of the generalized Bernoulli functions and of the generalized Euler functions of the first and second kinds as follows.

\begin{corollary}\label{cor3.3} Let $q\in\mathbb{N}$ and $x\in\mathbb{R}$. Let $\chi$ be a Dirichlet character modulo $q$. Then, for $n\in\mathbb{N}$,
\begin{equation}\label{eq3.23}
\overline{B}_{n}(x,\chi)=-\frac{q^{n-1}n!}{(2\pi \mathrm{i})^{n}}\sum_{k\in\mathbb{Z}\setminus\{0\}}\frac{G(k,\chi)e^{\frac{2\pi\mathrm{i}kx}{q}}}{k^{n}},
\end{equation}
for $n\in\mathbb{N}_{0}$,
\begin{equation}\label{eq3.24}
\overline{e}_{n}(x,\chi)=\frac{2q^{n}n!}{(2\pi \mathrm{i})^{n+1}}\sum_{k\in\mathbb{Z}}\frac{G(k,\chi)e^{\frac{2\pi\mathrm{i}(k-\frac{1}{2})x}{q}}}{(k-\frac{1}{2})^{n+1}},
\end{equation}
and for $n\in\mathbb{N}_{0}$ and $2\nmid q$,
\begin{equation}\label{eq3.25}
\overline{E}_{n}(x,\chi)=\frac{2q^{n}n!}{(2\pi \mathrm{i})^{n+1}}\sum_{k\in\mathbb{Z}}\frac{G(k,\chi)e^{\frac{2\pi\mathrm{i}(k-\frac{q}{2})x}{q}}}{(k-\frac{q}{2})^{n+1}}.
\end{equation}
\end{corollary}

\begin{proof}
Taking $\lambda=1$ in \eqref{eq3.1}, we get \eqref{eq3.23}. Taking $\lambda=e^{\frac{\pi\mathrm{i}}{q}}$ and $\lambda=-1$ in \eqref{eq3.1}, replacing $n$ by $n+1$, and then multiplying by
$-2/(n+1)$, we obtain \eqref{eq3.24} and \eqref{eq3.25}, respectively.
\end{proof}

As a result of Corollary \ref{cor3.3}, we have

\begin{corollary}\label{cor3.4} Let $q\in\mathbb{N}$ and $x\in\mathbb{R}$. Let $\chi$ be a primitive character modulo $q$. Then, for $n\in\mathbb{N}$,
\begin{equation}\label{eq3.26}
\overline{B}_{n}(x,\overline{\chi})=-\frac{q^{n-1}n!G(1,\overline{\chi})}{(2\pi \mathrm{i})^{n}}\sum_{k\in\mathbb{Z}\setminus\{0\}}\frac{\chi(k)e^{\frac{2\pi\mathrm{i}kx}{q}}}{k^{n}},
\end{equation}
for $n\in\mathbb{N}_{0}$,
\begin{equation}\label{eq3.27}
\overline{e}_{n}(x,\overline{\chi})=\frac{2q^{n}n!G(1,\overline{\chi})}{(2\pi \mathrm{i})^{n+1}}\sum_{k\in\mathbb{Z}}\frac{\chi(k)e^{\frac{2\pi\mathrm{i}(k-\frac{1}{2})x}{q}}}{(k-\frac{1}{2})^{n+1}},
\end{equation}
and for $n\in\mathbb{N}_{0}$ and $2\nmid q$,
\begin{equation}\label{eq3.28}
\overline{E}_{n}(x,\overline{\chi})=\frac{2q^{n}n!G(1,\overline{\chi})}{(2\pi \mathrm{i})^{n+1}}\sum_{k\in\mathbb{Z}}\frac{\chi(k)e^{\frac{2\pi\mathrm{i}(k-\frac{q}{2})x}{q}}}{(k-\frac{q}{2})^{n+1}}.
\end{equation}
\end{corollary}

\begin{proof} Since for a primitive character $\chi$ modulo $q$ we have the identity (see, e.g., \cite[Theorem 8.15]{apostol4})
\begin{equation}\label{eq3.29}
G(n,\chi)=\overline{\chi}(n)G(1,\chi)\quad(n\in\mathbb{Z}),
\end{equation}
applying it to \eqref{eq3.23}, \eqref{eq3.24} and \eqref{eq3.25} gives the desired results.
\end{proof}

Notably, using Euler's formula $e^{\mathrm{i}x}=\cos x+\mathrm{i}\sin x$ in \eqref{eq3.26}, we recover the generalized Bernoulli functions defined by Berndt \cite[Definition 1]{berndt3}.

Furthermore, Theorem \ref{thm3.1} yields the following result.

\begin{corollary}\label{cor3.5} Let $n,q\in\mathbb{N}$ and $x\in\mathbb{R}$. Let $\chi$ be a primitive character modulo $q$. Then, for $a\in\mathbb{Z}$ with $-\frac{q}{2}< a\leqslant\frac{q}{2}$,
\begin{eqnarray}\label{eq3.30}
&&\sum_{k=1}^{\infty}\frac{\chi(k+a)e^{\frac{2\pi \mathrm{i}kx}{q}}}{k^{n}}+(-1)^{n}\sum_{k=1}^{\infty}\frac{\chi(-k+a)e^{-\frac{2\pi \mathrm{i}kx}{q}}}{k^{n}}\nonumber\\
&&\qquad=-\frac{(2\pi\mathrm{i})^{n}\chi(-1)G(1,\chi)\overline{\beta}_{n,\overline{\chi}}(x,e^{\frac{2\pi\mathrm{i}a}{q}})}{q^{n}n!}.
\end{eqnarray}
\end{corollary}

\begin{proof}
Clearly, from \eqref{eq3.29} and the geometric sum stated in \cite[Theorem 8.1]{apostol4}, we have
\begin{eqnarray}\label{eq3.31}
G(1,\chi)G(1,\overline{\chi})&=&G(1,\chi)\sum_{r=1}^{q}\overline{\chi}(r)e^{\frac{2\pi \mathrm{i}r}{q}}\nonumber\\
&=&\sum_{r=1}^{q}G(r,\chi)e^{\frac{2\pi \mathrm{i}r}{q}}\nonumber\\
&=&\sum_{j=1}^{q}\chi(j)\sum_{r=1}^{q}e^{\frac{2\pi \mathrm{i}(j+1)r}{q}}\nonumber\\
&=&\chi(-1)q.
\end{eqnarray}
Hence, taking $\lambda=e^{\frac{2\pi\mathrm{i}a}{q}}$ in \eqref{eq3.1} and using \eqref{eq3.29} and \eqref{eq3.31}, we have
\begin{equation*}
\overline{\beta}_{n,\overline{\chi}}(x,e^{\frac{2\pi\mathrm{i}a}{q}})=-\frac{q^{n}n!\chi(-1)}{(2\pi \mathrm{i})^{n}G(1,\chi)}\sideset{}{'}\sum_{k=-\infty}^{+\infty}\frac{\chi(k)e^{\frac{2\pi\mathrm{i}(k-a)x}{q}}}{(k-a)^{n}},
\end{equation*}
as desired.
\end{proof}

In particular, the case $a=x=0$ in Corollary \ref{cor3.5} gives the following: For $n\in\mathbb{N}$ and $r\in\{0,1\}$, if $\chi(-1)=(-1)^{r}$ with $r$ and $n$ of the same parity, then for a primitive character $\chi$,
\begin{equation}\label{eq3.32}
L(n,\chi)=\frac{(-1)^{r+1}2^{n-1}\pi^{n}\mathrm{i}^{n}G(1,\chi)B_{n,\overline{\chi}}}{q^{n}n!},
\end{equation}
which is equivalent to the results shown in \cite[pp. 442--443]{neukirch} and \cite[Theorem 10.3.1]{cohen}. For other formulas of $L(n,\chi)$ at positive integers, the reader may consult the monograph of Shimura \cite{shimura}.

We next use Theorem \ref{thm3.1} to show that the values of Berndt's generalized $L$-function at integers can be explicitly evaluated in terms of the generalized Apostol-Bernoulli functions.

\begin{corollary}\label{cor3.6} Let $n,q\in\mathbb{N}$ with $q\geqslant2$ and $x\in\mathbb{R}$. Let $\chi$ be a primitive character modulo $q$. Then, for $a\in\mathbb{R}$ with $0\leqslant a\leqslant1$,
\begin{eqnarray}\label{eq3.33}
&&L(n,x,-a,\chi)+(-1)^{n}\chi(-1)L(n,-x,a,\chi)\nonumber\\
&&\qquad=-\frac{(2\pi\mathrm{i})^{n}e^{\frac{2\pi\mathrm{i}ax}{q}}\chi(-1)G(1,\chi)\overline{\beta}_{n,\overline{\chi}}(x,e^{\frac{2\pi\mathrm{i}a}{q}})}{q^{n}n!}.
\end{eqnarray}
\end{corollary}

\begin{proof}
Taking $\lambda=e^{\frac{2\pi\mathrm{i}a}{q}}$ in \eqref{eq3.1} and using \eqref{eq3.29} and \eqref{eq3.31}, we have
\begin{equation}\label{eq3.34}
e^{\frac{2\pi\mathrm{i}ax}{q}}\overline{\beta}_{n,\overline{\chi}}(x,e^{\frac{2\pi\mathrm{i}a}{q}})=-\frac{q^{n}n!\chi(-1)}{(2\pi \mathrm{i})^{n}G(1,\chi)}\sideset{}{'}\sum_{k=-\infty}^{+\infty}\frac{\chi(k)e^{\frac{2\pi\mathrm{i}kx}{q}}}{(k-a)^{n}}.
\end{equation}
It is clear that for $q\geq2$,
\begin{eqnarray*}
\sideset{}{'}\sum_{k=-\infty}^{+\infty}\frac{\chi(k)e^{\frac{2\pi\mathrm{i}kx}{q}}}{(k-a)^{n}}&=&\sideset{}{'}\sum_{k=1}^{+\infty}\frac{\chi(k)e^{\frac{2\pi\mathrm{i}kx}{q}}}{(k-a)^{n}}+
\sideset{}{'}\sum_{k=0}^{+\infty}\frac{\chi(-k)e^{\frac{-2\pi\mathrm{i}kx}{q}}}{(-k-a)^{n}}\nonumber\\
&=&\sideset{}{'}\sum_{k=0}^{+\infty}\frac{\chi(k)e^{\frac{2\pi\mathrm{i}kx}{q}}}{(k-a)^{n}}+(-1)^{n}\chi(-1)
\sideset{}{'}\sum_{k=0}^{+\infty}\frac{\chi(k)e^{\frac{-2\pi\mathrm{i}kx}{q}}}{(k+a)^{n}}.
\end{eqnarray*}
Thus, by inserting the above identity into \eqref{eq3.34}, we obtain the desired result.
\end{proof}

\begin{corollary}\label{cor3.7} Let $n,q\in\mathbb{N}$ with $q\geqslant2$ and let $\chi$ be a primitive character modulo $q$. Then, for $x,a\in\mathbb{R}$ with $0\leqslant x\leqslant1$ and $0\leqslant a\leqslant1$,
\begin{equation}\label{eq3.35}
L(1-n,x,a,\chi)=-\frac{\overline{\beta}_{n,\chi}(a,e^{\frac{2\pi\mathrm{i}x}{q}})}{n}.
\end{equation}
\end{corollary}

\begin{proof}
Taking $s=n$ in Berndt's functional equation \eqref{eq1.17}, we have
\begin{eqnarray*}
L(1-n,x,a,\chi)&=&\frac{q^{n-1}(n-1)!}{(2\pi \mathrm{i})^{n}}G(1,\chi)e^{-\frac{2\pi \mathrm{i}ax}{q}}\nonumber\\
&&\times\bigl(L(n,a,-x,\overline{\chi})+(-1)^{n}\chi(-1)L(n,-a,x,\overline{\chi})\bigl).
\end{eqnarray*}
Therefore, by applying Corollary \ref{cor3.6} to the right-hand side of the above identity, with the help of \eqref{eq3.31}, we prove Corollary \ref{cor3.7}.
\end{proof}

To conclude this paper, we use Theorem \ref{thm3.1} to derive a result analogous to Proposition \ref{pro2.4} as follows.

\begin{theorem}\label{thm3.8} Let $q,a,b\in\mathbb{N}$ and $\lambda,\mu\in\mathbb{C}\setminus\{0\}$ with $\lambda^{a}=\mu^{b}$. Let $\chi$ be a primitive character modulo $q$. Then, for $n\in\mathbb{N}$,
\begin{equation}\label{eq3.36}
a^{n-1}\sum_{k=0}^{aq-1}\chi(k)\lambda^{k}\overline{\beta}_{n,\chi}\biggl(bx+\frac{bk}{a},\mu\biggl)
=b^{n-1}\sum_{k=0}^{bq-1}\chi(k)\mu^{k}\overline{\beta}_{n,\chi}\biggl(ax+\frac{ak}{b},\lambda\biggl).
\end{equation}
\end{theorem}

\begin{proof}
Since $\lambda^{a}=\mu^{b}$, there exists $r\in\mathbb{Z}$ such that $a\log\lambda=b\log\mu+2\pi\mathrm{i}r$. So from \eqref{eq3.1}, we have
\begin{eqnarray}\label{eq3.37}
&&a^{n-1}\sum_{k=0}^{aq-1}\chi(k)\lambda^{k}\overline{\beta}_{n,\chi}\biggl(bx+\frac{bk}{a},\mu\biggl)\nonumber\\
&&\qquad=-\frac{a^{n-1}q^{n-1}n!}{e^{bx\log\mu}(2\pi \mathrm{i})^{n}}\sideset{}{'}\sum_{l=-\infty}^{+\infty}\frac{G(l,\chi)e^{\frac{2\pi\mathrm{i}blx}{q}}}{(l-\frac{q\log\mu}{2\pi\mathrm{i}})^{n}}\sum_{k=1}^{aq}\chi(-k)e^{-\frac{2\pi\mathrm{i}(bl+rq)k}{aq}}.
\end{eqnarray}
Note that from the geometric sum we have
\begin{eqnarray}\label{eq3.38}
&&\sum_{k=1}^{aq}\chi(-k)e^{-\frac{2\pi\mathrm{i}(bl+rq)k}{aq}}\nonumber\\
&&\qquad=\sum_{k=1}^{q}\chi(-k)e^{-\frac{2\pi\mathrm{i}(bl+rq)k}{aq}}+e^{-\frac{2\pi\mathrm{i}(bl+rq)}{a}}
\sum_{k=1}^{q}\chi(-k)e^{-\frac{2\pi\mathrm{i}(bl+rq)k}{aq}}+\cdots\nonumber\\
&&\qquad\quad+e^{-\frac{2\pi\mathrm{i}(bl+rq)(a-1)}{a}}\sum_{k=1}^{q}\chi(-k)e^{-\frac{2\pi\mathrm{i}(bl+rq)k}{aq}}\nonumber\\
&&\qquad=\sum_{k=1}^{q}\chi(-k)e^{-\frac{2\pi\mathrm{i}(bl+rq)k}{aq}}\sum_{j=0}^{a-1}e^{-\frac{2\pi\mathrm{i}(bl+rq)j}{a}}\nonumber\\
&&\qquad= a\epsilon(a,bl+rq)\sum_{k=1}^{q}\chi(-k)e^{-\frac{2\pi\mathrm{i}(bl+rq)k}{aq}},
\end{eqnarray}
where $\epsilon(a,l)=1$ or $0$ according to $a\mid l$ or $a\nmid l$. Hence, inserting \eqref{eq3.38} into \eqref{eq3.37}, we obtain from \eqref{eq3.29} that
\begin{eqnarray}\label{eq3.39}
&&a^{n-1}\sum_{k=0}^{aq-1}\chi(k)\lambda^{k}\overline{\beta}_{n,\chi}\biggl(bx+\frac{bk}{a},\mu\biggl)\nonumber\\
&&\qquad=-\frac{a^{n}q^{n-1}n!G(1,\chi)^{2}}{e^{bx\log\mu}(2\pi \mathrm{i})^{n}}\sideset{}{'}\sum_{\substack{l=-\infty\\a\mid (bl+rq)}}^{+\infty}\frac{\overline{\chi}(l)\overline{\chi}(\frac{bl+rq}{a})e^{\frac{2\pi\mathrm{i}blx}{q}}}{(l-\frac{q\log\mu}{2\pi\mathrm{i}})^{n}}.
\end{eqnarray}
In a consideration to \eqref{eq3.39}, we find that
\begin{eqnarray}\label{eq3.40}
&&b^{n-1}\sum_{k=0}^{bq-1}\chi(k)\mu^{k}\overline{\beta}_{n,\chi}\biggl(ax+\frac{ak}{b},\lambda\biggl)\nonumber\\
&&\qquad=-\frac{b^{n}q^{n-1}n!G(1,\chi)^{2}}{e^{ax\log\lambda}(2\pi \mathrm{i})^{n}}\sideset{}{'}\sum_{\substack{l=-\infty\\b\mid (al-rq)}}^{+\infty}\frac{\overline{\chi}(l)\overline{\chi}(\frac{al-rq}{b})e^{\frac{2\pi\mathrm{i}alx}{q}}}{(l-\frac{q\log\lambda}{2\pi\mathrm{i}})^{n}}.
\end{eqnarray}
Since the Diophantine equation $ax+by=rq$ is solvable if and only if $\gcd(a,b)\mid rq$, and in the case when it is solvable, all solutions of $ax+by=rq$ are given by
\begin{equation*}
x=x_{0}+\frac{bd}{\gcd(a,b)},\quad y=y_{0}-\frac{ad}{\gcd(a,b)},
\end{equation*}
where $x_{0},y_{0},d\in\mathbb{Z}$ with $ax_{0}+by_{0}=rq$. Let $\gcd(a,b)=c$. It follows from \eqref{eq3.39} and \eqref{eq3.40} that
\begin{eqnarray*}
&&a^{n-1}\sum_{k=0}^{aq-1}\chi(k)\lambda^{k}\overline{\beta}_{n,\chi}\biggl(bx+\frac{bk}{a},\mu\biggl)\nonumber\\
&&\qquad=-\frac{a^{n}b^{n}q^{n-1}n!G(1,\chi)^{2}}{e^{bx\log\mu}(2\pi \mathrm{i})^{n}}\sideset{}{'}\sum_{\substack{l=-\infty\\a\mid (bl+rq)}}^{+\infty}\frac{\overline{\chi}(l)\overline{\chi}(\frac{bl+rq}{a})e^{\frac{2\pi\mathrm{i}blx}{q}}}{(bl+rq-\frac{qa\log\lambda}{2\pi\mathrm{i}})^{n}}\nonumber\\
&&\qquad=-\frac{b^{n}q^{n-1}n!G(1,\chi)^{2}}{e^{bx\log\mu}(2\pi \mathrm{i})^{n}}\sideset{}{'}\sum_{d=-\infty}^{+\infty}\frac{\overline{\chi}(-y_{0}-\frac{ad}{c})\overline{\chi}(x_{0}-\frac{bd}{c})e^{\frac{2\pi\mathrm{i}bx(-y_{0}-\frac{ad}{c})}{q}}}{(x_{0}-\frac{bd}{c}-\frac{q\log\lambda}{2\pi\mathrm{i}})^{n}}\nonumber\\
&&\qquad=-\frac{b^{n}q^{n-1}n!G(1,\chi)^{2}}{e^{ax\log\lambda}(2\pi \mathrm{i})^{n}}\sideset{}{'}\sum_{d=-\infty}^{+\infty}\frac{\overline{\chi}(-y_{0}-\frac{ad}{c})\overline{\chi}(x_{0}-\frac{bd}{c})e^{\frac{2\pi\mathrm{i}ax(x_{0}-\frac{bd}{c})}{q}}}{(x_{0}-\frac{bd}{c}-\frac{q\log\lambda}{2\pi\mathrm{i}})^{n}}.
\end{eqnarray*}
and
\begin{eqnarray*}
&&b^{n-1}\sum_{k=0}^{bq-1}\chi(k)\mu^{k}\overline{\beta}_{n,\chi}\biggl(ax+\frac{ak}{b},\lambda\biggl)\nonumber\\
&&\qquad=-\frac{b^{n}q^{n-1}n!G(1,\chi)^{2}}{e^{ax\log\lambda}(2\pi \mathrm{i})^{n}}\sideset{}{'}\sum_{d=-\infty}^{+\infty}\frac{\overline{\chi}(x_{0}-\frac{bd}{c})\overline{\chi}(-y_{0}-\frac{ad}{c})e^{\frac{2\pi\mathrm{i}ax(x_{0}-\frac{bd}{c})}{q}}}{(x_{0}-\frac{bd}{c}-\frac{q\log\lambda}{2\pi\mathrm{i}})^{n}}.
\end{eqnarray*}
Thus, equating the above two identities, we obtain \eqref{eq3.36} and finish the proof of Theorem \ref{thm3.8}.
\end{proof}

We next discuss some special cases of Theorem \ref{thm3.8}. We obtain the following results for the Apostol-Bernoulli functions and the Apostol-Euler functions.

\begin{corollary}\label{cor3.9} Let $a,b\in\mathbb{N}$ and $\lambda,\mu\in\mathbb{C}\setminus\{0\}$. If $\lambda^{a}=\mu^{b}$, then for $n\in\mathbb{N}$,
\begin{equation}\label{eq3.41}
a^{n-1}\sum_{k=0}^{a-1}\lambda^{k}\overline{\beta}_{n}\biggl(bx+\frac{bk}{a},\mu\biggl)=b^{n-1}\sum_{k=0}^{b-1}\mu^{k}\overline{\beta}_{n}\biggl(ax+\frac{ak}{b},\lambda\biggl),
\end{equation}
and if $(-\lambda)^{a}=(-\mu)^{b}$, then for $n\in\mathbb{N}_{0}$,
\begin{equation}\label{eq3.42}
a^{n}\sum_{k=0}^{a-1}(-\lambda)^{k}\overline{\varepsilon}_{n}\biggl(bx+\frac{bk}{a},\mu\biggl)=b^{n}\sum_{k=0}^{b-1}(-\mu)^{k}\overline{\varepsilon}_{n}\biggl(ax+\frac{ak}{b},\lambda\biggl).
\end{equation}
\end{corollary}

\begin{proof}
Setting $q=1$ in \eqref{eq3.36} gives \eqref{eq3.41}. Similarly, setting $q=1$ in \eqref{eq3.36}, replacing $\lambda$ by $-\lambda$, $\mu$ by $-\mu$ and $n$ by $n+1$, and then multiplying by
$-2/(n+1)$ yields \eqref{eq3.42}.
\end{proof}

In particular, taking $\lambda=\mu=b=1$ in \eqref{eq3.41} and \eqref{eq3.42} gives that for $n,a\in\mathbb{N}$,
\begin{equation}\label{eq3.43}
a^{n-1}\sum_{k=0}^{a-1}\overline{B}_{n}\biggl(x+\frac{k}{a}\biggl)=\overline{B}_{n}(ax),
\end{equation}
and for $n\in\mathbb{N}_{0}$ and $a\in\mathbb{N}$ with $2\nmid a$,
\begin{equation}\label{eq3.44}
a^{n}\sum_{k=0}^{a-1}(-1)^{k}\overline{E}_{n}\biggl(x+\frac{k}{a}\biggl)=\overline{E}_{n}(ax).
\end{equation}
If we take $\lambda=-1$ and $\mu=1$ in \eqref{eq3.41}, then for $n,a\in\mathbb{N}$ with $2\mid a$,
\begin{equation}\label{eq3.45}
a^{n-1}\sum_{k=0}^{a-1}(-1)^{k}\overline{B}_{n}\biggl(x+\frac{k}{a}\biggl)=-\frac{n\overline{E}_{n-1}(ax)}{2}.
\end{equation}
Formula \eqref{eq3.43} is due to Raabe \cite{raabe}, is usually called Raabe's formula, and plays a fundamental role in the theory of Dedekind sums (see, for example, \cite{bayad1,carlitz,hall,he2}).

We also obtain the following results for the generalized Bernoulli functions and the generalized Euler functions of the first and second kinds.

\begin{corollary}\label{cor3.10} Let $q,a,b\in\mathbb{N}$. Let $\chi$ be a primitive character modulo $q$. Then, for $n\in\mathbb{N}$,
\begin{equation}\label{eq3.46}
a^{n-1}\sum_{k=0}^{aq-1}\chi(k)\overline{B}_{n}\biggl(bx+\frac{bk}{a},\chi\biggl)=b^{n-1}\sum_{k=0}^{bq-1}\chi(k)\overline{B}_{n}\biggl(ax+\frac{ak}{b},\chi\biggl).
\end{equation}
Furthermore, if $e^{\frac{\pi \mathrm{i}a}{q}}=e^{\frac{\pi \mathrm{i}b}{q}}$, then for $n\in\mathbb{N}_{0}$,
\begin{equation}\label{eq3.47}
a^{n}\sum_{k=0}^{aq-1}\chi(k)e^{\frac{\pi \mathrm{i}k}{q}}\overline{e}_{n}\biggl(bx+\frac{bk}{a},\chi\biggl)=b^{n}\sum_{k=0}^{bq-1}\chi(k)e^{\frac{\pi \mathrm{i}k}{q}}\overline{e}_{n}\biggl(ax+\frac{ak}{b},\chi\biggl),
\end{equation}
and if $(-1)^{a}=(-1)^{b}$ and $2\nmid q$, then for $n\in\mathbb{N}_{0}$,
\begin{equation}\label{eq3.48}
a^{n}\sum_{k=0}^{aq-1}\chi(k)(-1)^{k}\overline{E}_{n}\biggl(bx+\frac{bk}{a},\chi\biggl)=b^{n}\sum_{k=0}^{bq-1}\chi(k)(-1)^{k}\overline{E}_{n}\biggl(ax+\frac{ak}{b},\chi\biggl).
\end{equation}
\end{corollary}

\begin{proof}
Taking $\lambda=\mu=1$ in \eqref{eq3.36}, we get \eqref{eq3.46}. Taking $\lambda=\mu=e^{\frac{\pi\mathrm{i}}{q}}$ and $\lambda=\mu=-1$ in \eqref{eq3.36}, replacing $n$ by $n+1$, and then multiplying by
$-2/(n+1)$, we obtain \eqref{eq3.47} and \eqref{eq3.48}, respectively.
\end{proof}

\end{document}